\documentclass[10pt, english,oneside,reqno]{amsart}

\usepackage[utf8]{inputenc}
\usepackage[english]{babel}
\usepackage{hyperref}
\usepackage{color}
\usepackage{enumerate}
\usepackage{fancyhdr}
\usepackage[utf8]{inputenc}

\usepackage{amsmath}
\usepackage{amscd}
\usepackage{amsfonts}
\usepackage{amsthm}
\usepackage{amstext}
\usepackage{amssymb}
\usepackage{amsbsy}
\usepackage{mathrsfs}
\usepackage{tikz-cd}
\usepackage{mathtools}
\usepackage{enumitem}

\usepackage{float}
\usepackage{caption}
\usepackage{subcaption}
\usepackage{graphicx}

\usepackage[alphabetic,abbrev]{amsrefs}

\usepackage{url}

\theoremstyle{plain}
\newtheorem{theorem}{Theorem}[section]

\newtheorem{lemma}[theorem]{Lemma}

\theoremstyle{remark}

\newtheorem{remark}[theorem]{Remark}
\numberwithin{equation}{section}
\numberwithin{figure}{section}

\newcommand{\eps}{\varepsilon}

\newcommand{\R}{\mathbb{R}}

\newcommand{\E}{\mathbb{E}}

\newcommand{\Haarof}[1]{m_{#1}}

\begin{document}
	\title[Polynomial Tail Decay for Stationary Measures]{Polynomial Tail Decay for Stationary Measures}

    \begin{abstract}
        We show on complete metric spaces a polynomial tail decay for stationary measures of contracting on average generating measures.  
    \end{abstract}

    \author[sk]{Samuel Kittle}
    \author[ck]{Constantin Kogler}

    \email{s.kittle@ucl.ac.uk, kogler@ias.edu}

    \address{Samuel Kittle, Department of Mathematics, University College London, 25 Gordon Street, London WC1H 0AY, United Kingdom}

    \address{Constantin Kogler, Institute for Advanced Study, 1 Einstein Dr, Princeton, NJ 08540, United States of America}

    \maketitle

    \tableofcontents
	
    \section{Introduction}

   Let $X$ be a complete metric space. For a continuous map $g: X \to X$ denote by $\rho(g)$ the Lipschitz constant of $g$ defined as $$\rho(g) = \min\{ \rho > 0 \,:\, d(gx, gy) \leq \rho\cdot d(x,y) \text{ for all } x,y\in X \},$$ where we use throughout this paper the notation $gx = g(x)$ for all $x\in X$.  Consider the semi-group $L(X) = \{ g: X \to X \,:\, \rho(g) < \infty\}$ of Lipschitz maps. For a probability measure $\mu$ on $L(X)$ we define the \textbf{contraction rate} of $\mu$ as $$\chi_{\mu} = \E_{g \sim \mu}[\log\rho(g)] = \int \log \rho(g) \, d\mu(g),$$ whenever it exists. If a probability measure $\mu$ on $L(X)$ satisfies $\chi_{\mu} < 0$, we call it \textbf{contracting on average}.
   
   Given a further probability measure $\nu$ on $X$, the convolution $\mu * \nu$ is the measure uniquely determined by satisfying $$(\mu*\nu)(f) = \int\int f(gx) \, d\mu(g)d\nu(x)$$ for all continuous compactly supported functions $f: X \to \R$.
   
    The space $L(X)$ is endowed with the compact open topology, that is, the topology of uniform convergence on compact spaces. The main result is the following. We refer to the end of the introduction for the asymptotic notation used. 

\begin{theorem}\label{HutchinsonSuperGen1}
    Let $X$ be a complete metric space and let $\mu$ be a compactly supported probability measure on $L(X)$ with $\chi_{\mu} < 0$ and satisfying $\mathbb{E}_{g\sim \mu}[\rho(g)^t] < \infty$ for all $t \in \R$. Then there exists a unique probability measure $\nu$ on $X$ such that $\mu * \nu = \nu$. Moreover, there is $\alpha = \alpha(\mu) > 0$ such that for all $R > 0$ and $x \in X$,
    \begin{equation}\label{Concentration}
        \nu(\{ y \in X \,:\, d(x,y) \geq R \}) \ll_{\mu,x} R^{-{\alpha}},
    \end{equation} where the implied constant depends on $\mu$ and $x$.
\end{theorem}

    Theorem~\ref{HutchinsonSuperGen1} applies in particular when $\mu$ is finitely supported. The assumption that $\mathbb{E}_{g\sim \mu}[\rho(g)^t] < \infty$ implies when $t < 0$ that $\rho(g) > 0$ for every $g \in \mathrm{supp}(\mu)$. We remark that we do not require $X$ to be locally compact or separable, yet we observe that when $\mu$ is countably supported, it follows from our proof that $\nu$ is supported on a separable subset. The assumption that $\mu$ is compactly supported does not imply that $\mathbb{E}_{g\sim \mu}[\rho(g)^t] < \infty$ for any $t \in \R$, as we show in section~\ref{L(X)example}.
    
    As discussed in section~\ref{section:PolynomialLower}, when $X = \R^d$ and $\nu$ is a self-similar measure, we can easily deduce a polynomial lower bound for $\nu(\{ x \in \R^d \,:\, |x| \geq R \})$ provided there is some $g\in \mathrm{supp}(\mu)$ such that $\rho(g) > 1$. The authors were initially motivated to prove \eqref{Concentration} in their work on absolutely continuous self-similar measures \cite{KittleKoglerAC} to show that smoothenings of contracting on average self-similar measures have finite differential entropy, which follows from a similar argument to Lemma~\ref{FiniteDiffEnt}.
    
    The existence and uniqueness of $\nu$ is well known in numerous cases (cf., for example, \cite{Hutchinson1981}, \cite{Barnselyetal1988}, \cite{Stenflo2012}), so the main novelty lies in \eqref{Concentration}. The quantity $ \nu(\{ y \in X \,: \, d(x,y) \geq R\})$ as $R$ grows is called the tail of a stationary measure, and we refer to \eqref{Concentration} as polynomial tail decay. We will comment on what is known about tail decay results similar to \eqref{Concentration} below, yet first state a version of Theorem~\ref{HutchinsonSuperGen1} using the Lyapunov exponent instead of the contraction rate. Indeed, we define the Lyapunov exponent of $\mu$ as $$\lambda_{\mu} = \lim_{n \to \infty}\frac{\E_{g \sim \mu^{*n}}[\log \rho(g)]}{n}  = \inf_{n \geq 1}\frac{\E_{g \sim \mu^{*n}}[\log \rho(g)]}{n},$$ whenever it exists, where the limit and infimum are equal by Fekete's lemma. 
    
    The reader may observe that $\lambda_{\mu} \leq \chi_{\mu}$. The latter inequality is often strict, as the following example shows. We consider the $\mathrm{GL}_d(\R)$ action on $\R^d$ and let $\mu = \delta_A$ for $A \in \mathrm{GL}_d(\R)$. In this setting, $\rho(A) = ||A||_{\mathrm{op}}$ with the operator norm being equal to the maximal singular value of $A$ and $\lambda_{\mu}$ is equal to the logarithm of the spectral radius of $A$. Choosing then, for example, $d = 2$ and $A = (\begin{smallmatrix} 1 & 1 \\ 0 & 1 \end{smallmatrix})$, we have $\lambda_{\mu} = 0$ while $\chi_{\mu} = \frac{1}{2}\log(\tfrac{3 + \sqrt{5}}{2}) > 0$.

    We can establish the conclusion of Theorem~\ref{HutchinsonSuperGen1} by assuming that $\lambda_{\mu} < 0$ and requiring a large deviation principle for $\rho$. To introduce notation, denote by 
    \begin{equation}\label{gammadef}
        \gamma_1, \gamma_2, \ldots
    \end{equation}
    independent $\mu$-distributed random variables on $L(X)$ sampled from the probability space $(\Omega, \mathscr{F}, \mathbb{P})$.

    \begin{theorem}\label{HutchinsonLyapunovGen}
        Let $X$ be a complete metric space and let $\mu$ be a compactly supported probability measure on $L(X)$ with $\lambda_{\mu} < 0$. Assume further that for every $\eps > 0$ there is some $\delta > 0$ such that for sufficiently large $n$, 
        \begin{equation}\label{LDPAssump}
            \mathbb{P}\left[ \big| n\lambda_{\mu}  - \log \rho(\gamma_1\cdots \gamma_n)  \big| > \eps n   \right] \leq e^{-\delta n}.
        \end{equation}
        Then the conclusion of Theorem~\ref{HutchinsonSuperGen1} holds.
    \end{theorem}

    Consider for $d \geq 1$ the group of affine transformations $\mathrm{Aff}(\R^d)$ of $\R^d$, that is the maps $x \mapsto Ax + b$ with $A \in \mathrm{GL}_d(\R)$ and $b\in \R^d$. If $\mu$ is a finitely supported probability measure on $\mathrm{Aff}(\R^d)$ with unique stationary measure $\nu$, then $\nu$ is called a \textbf{self-affine measure}. 
    
    Much of the literature on tail estimates of stationary measures concerns self-affine measures. Indeed, as the authors learned after proving Theorem~\ref{HutchinsonSuperGen1}, polynomial tail decay is known for self-affine measures and follows, for example, from the moment estimates of \cite{GuivarchLePage2016}*{Proposition 5.1}.

    We now comment on previous results on the tail behaviour of self-affine measures. Let $\mu$ be a probability measure on $\mathrm{Aff}(\R^d)$ and denote by $A(\mu)$ the pushforward of $\mu$ under $g \mapsto A(g)$ with $g(x) = A(g)x + b(g)$ for $A(g) \in \mathrm{GL}_d(\R)$ and $b(g) \in \R^d$ for all $x \in \R^d$. Note that the Lipschitz constant of an element $g \in \mathrm{Aff}(\R^d)$ is the operator norm $||A(g)||_{\mathrm{op}}$. For simplicity, assume in the following that $\mu$ is compactly supported. 
    
    In pioneering work, Kesten \cite{Kesten1973} developed a renewal theory for random matrix products. As an application, it is shown in \cite[Theorem B]{Kesten1973} that if (among some further assumptions) all entries of the matrices in the support of $A(\mu)$ are $>0$ and the set of logarithms of the largest eigenvalue of the products of the matrices in the support of $A(\mu)$ is dense in $\R$, then the limit $R^{\alpha}\nu(\{ x\in \R^d \,:\, |x| \geq R \})$ exists as $R \to \infty$ for some $\alpha > 0$. For $d = 1$, Kesten's result was partially rediscovered by Grincevi\v{c}ius \cite{Grincevicjus1975}, who also dealt with the case that $\log A(g)$ for $g \in \mathrm{supp}(\mu)$ does not generate a dense subgroup. In fact, in \cite{Grincevicjus1975}*{Assertion 1, Page 9} a version of Theorem~\ref{HutchinsonSuperGen1} is shown in the self-affine case for $d = 1$ under the assumption that $A(g) > 0$ for all $g \in \mathrm{supp}(\mu)$. Goldie \cite{Goldie1991} used the methods of \cite{Grincevicjus1975} to generalise Kesten's result in dimension one. Based on Goldie's results, Kevei \cite{Kevei2016} gave fine estimates for $\nu(\{ x\in \R \,:\, |x| \geq R \})$ under weak assumptions when $d = 1$. 
    
    Denote by $\Gamma_{A(\mu)} = \langle \mathrm{supp} \, A(\mu) \rangle < \mathrm{GL}_d(\R)$ the semi-group generated by the support of $A(\mu)$. Guivarc'h-Le Page \cite{GuivarchLePage2016}*{Theorem C} proved that when $d \geq 2$, $\lambda_{\mu} < 0$ and $\mathrm{supp}(\mu)$ does not have a common fixed point as well as if $\Gamma_{A(\mu)}$ is strongly irreducible and contains a proximal element (i.e. an element with a unique simple dominant eigenvalue), then $\lim_{R \to \infty} R^{-\alpha}(R \cdot \nu)$ converges to a limit Radon measure in the vague topology on $\mathbb{R}^d \backslash \{ 0 \}$, where $(R \cdot \nu)$ is the measure given by $(R \cdot \nu)(B) = \nu(R\cdot B)$ for $B$ a measurable set in $\mathbb{R}^d \backslash \{ 0 \}$ and $R\cdot B = \{R \cdot b \,:\, b \in B \}$. When $d = 1$, the same conclusion is shown under the additional assumption that $\Gamma_{A(\mu)}$ is not contained in a subgroup of $\R^*$. There are several further noticeable results in \cite{GuivarchLePage2016}, for example, also dealing with the case when the Lyapunov exponent is positive. It would be an interesting direction to generalise these results to self-similar measures, or to understand more precisely how the proximality assumption can be weakened.

    Kloeckner \cite{Kloeckner2022} proved general moment estimates for invariant function systems, from which tail estimates can be deduced. 

    Compared to these results, Theorem~\ref{HutchinsonSuperGen1} and Theorem~\ref{HutchinsonLyapunovGen} establish a coarse tail bound that holds under weak assumptions for general complete metric spaces.  Furthermore, we do not rely on renewal theory and only exploit the large deviation principle. 

    \subsection*{Notation}

    We use the asymptotic notation $A \ll B$ or $A = O(B)$ to denote that $|A| \leq CB$ for a constant $C > 0$. If the constant $C$ depends on additional parameters, we add subscripts. Moreover, $A \asymp B$ denotes $A \ll B$ and $B \ll A$. 

    For $x \in X$ and $R > 0$ we denote by $B_R(x)$ the open $R$-ball around $x$. When $X = \R^d$, we furthermore write $B_R = B_R(0)$. For a subset $B \subset X$ we denote by $B^c = X\backslash B$ the complement of $B$ in $X$.

     \subsection*{Acknowledgment} The first-named author gratefully acknowledges support from the Heilbronn Institute for Mathematical Research. This work was conducted during the second-named author's doctoral studies at the University of Oxford. We thank Timothée Bénard for pointing out that \eqref{Concentration} follows for self-affine measures from \cite{GuivarchLePage2016} and the referees for their thorough comments.

    \section{Proof of Theorem~\ref{HutchinsonSuperGen1} and Theorem~\ref{HutchinsonLyapunovGen}}

    \label{section:ProofMain}

In this section we prove Theorem~\ref{HutchinsonSuperGen1} and Theorem~\ref{HutchinsonLyapunovGen}. While existence and uniqueness of the stationary measure $\nu$ follows along standard lines, we include a short proof as the argument is needed to establish the polynomial tail decay of $\nu$. The latter follows from first showing that $\nu$ is well-approximated by $\mu^{*n}*\delta_x$ and then applying the large deviation principle. Indeed, we will use the following lemma that follows from Cramer's theorem. We denote, as in \eqref{gammadef}, by $\gamma_1, \gamma_2, \ldots$ independent $\mu$-distributed random elements of $L(X)$ sampled from the probability space $(\Omega, \mathscr{F}, \mathbb{P})$. 

\begin{lemma}\label{rhoLDP} Let $\mu$ be a probability measure on $L(X)$ satisfying $\mathbb{E}_{g\sim \mu}[\rho(g)^t] < \infty$ for all $t \in \R$. Then for every $\eps > 0$ there is $\delta = \delta(\mu,\eps) > 0$ such that for all sufficiently large $n$,
\begin{equation*}
    \mathbb{P}\Big[ \, | n\chi_{\mu} - \log \rho(\gamma_1) \cdots \rho(\gamma_n) | > \eps n \,\Big] \leq e^{-\delta n}.
\end{equation*}
\end{lemma}

\begin{proof}
    This follows from Cramér's Theorem \cite{KlenkeBook}*{Theorem 23.3} applied to $X_i = \log \rho(\gamma_i)$.
\end{proof}

\begin{lemma}\label{lambdaLemma}
    Let $\mu$ be a probability measure on $L(X)$ with $\chi_{\mu} < 0$ and satisfying $\mathbb{E}_{g\sim \mu}[\rho(g)^t] < \infty$ for all $t \in \R$. Then there is $\lambda \in (0,1)$ such that for almost all $\omega \in \Omega$ and $n \geq 1$ sufficiently large (depending on $\omega$), $$\rho(\gamma_1(\omega) \cdots \gamma_n(\omega)) \leq \rho(\gamma_1(\omega)) \cdots \rho(\gamma_n(\omega)) \leq \lambda^n.$$  
\end{lemma}

\begin{proof}
    The claim follows by applying Lemma~\ref{rhoLDP} to a sufficiently small $\eps > 0$ and the Borel-Cantelli Lemma.
\end{proof}

For notational convenience, write $$A_x = \sup_{\gamma\in \mathrm{supp}(\mu)} d(x,\gamma x),$$ which is finite since $\mu$ is compactly supported. In the proof of Theorem~\ref{HutchinsonSuperGen1}, we can drop the assumption of $\mu$ being compactly supported as
long as $A_x$ is finite for every $x\in X$.

\begin{proof} (of existence and uniqueness in Theorem~\ref{HutchinsonSuperGen1}) Given $x \in X$ and $\omega \in \Omega$ write $$z_n(x, \omega) = \gamma_1(\omega)\cdots \gamma_n(\omega) x.$$ We show that $\mathbb{P}$-almost surely, $z_n(x, \omega)$ is a Cauchy sequence. Indeed, let $\lambda \in (0,1)$ be as in Lemma~\ref{lambdaLemma}. Then almost surely for $k$ sufficiently large $d(z_k(x, \omega), z_{k + 1}(x, \omega)) \leq A_x\lambda^k$ and thus for sufficiently large $n$ and $m$, 
\begin{align}
    d(z_{n}(x, \omega),z_{m}(x,\omega)) &\leq \sum_{k = \min\{n,m \}}^{\max\{n,m\} - 1} d(z_k(x,\omega), z_{k + 1}(x,\omega)) \nonumber \\ &\leq A_x\sum_{k = \min\{n,m \}}^{\infty} \lambda^k = A_x \frac{\lambda^{\min\{n,m\}}}{1-\lambda}, \label{CauchyConvergence}
\end{align}
 which goes to zero as $\min\{n,m\}\to \infty$. Therefore, since $X$ is complete, the limit $\lim_{n \to \infty} z_n(x,\omega)$ exists for almost all $\omega \in \Omega$. The latter limit does not depend on $x$ as for almost all $\omega$ and sufficiently large $n$, $d(z_n(x,\omega), z_n(y,\omega)) \leq \lambda^n d(x,y),$ which goes to zero. Thus, there is a random variable $z:\Omega \to X$ such that for almost all $\omega$, 
 \begin{equation}\label{zdef}
     z(\omega) = \lim_{n \to \infty} z_n(x,\omega) = \lim_{n \to \infty} \gamma_1(\omega)\cdots \gamma_n(\omega)x
 \end{equation}
for all $x \in X$. 

 Let $\nu$ be the distribution of $z$. The measure $\nu$ is stationary since for any continuous bounded function $f$ on $X$, by dominated convergence,
 \begin{align*}
     (\mu*\nu)(f) &= \int\int f(gz) \, d\mu(g)d\nu(z) \\ &=  \int\int f(gz(\omega)) \, d\mu(g)d\mathbb{P}(\omega) \\
     &= \lim_{n \to \infty} \int\int f(gz_n(x,\omega)) \, d\mu(g)d\mathbb{P}(\omega) \\
     &= \lim_{n \to \infty} \int f(z_{n + 1}(x,\omega)) \, d\mathbb{P}(\omega) = \nu(f)
 \end{align*} for any $x\in X$.

 To finally show that the stationary measure is unique, let $\eta$ be a further stationary measure. Then for all $n\geq 1$ and any continuous bounded function $f$ on $X$, \begin{align*}
     \int f(x) \, d\eta(x) &=  \int\int f(gx) \, d\mu^{*n}(g)d\eta(x) = \int\int f(z_{n}(x,\omega)) \, d\mathbb{P}(\omega)d\eta(x).
 \end{align*} Letting $n \to \infty$, the right hand side tends to $\nu(f)$ by dominated convergence.
\end{proof}

\begin{remark}
 We have only used in the above proof that probability measures are uniquely characterised as positive linear functionals on $C_b(X)$ (the space of continuous bounded functions), which does not require $X$ to be locally compact or separable. Indeed, to prove this claim, note that given a subset $C \subset X$ we have $|d(x,C) - d(y,C)| \leq d(x,y)$ for $x,y \in X$ and therefore $x\mapsto d(x,C)$ is Lipschitz. If $C$ is measurable, it holds by continuity from above for a probability measure $\eta$ on $X$ that $\eta(C) = \lim_{n \to \infty} \eta(f_n)$ for $f_n(x) = r_n(d(x,C))$ with $x \in X$ and $r_n: \R_{\geq 0} \to \R_{\geq 0}$ to be defined for $n\geq 1$ as $$r_n(t) = \begin{cases}
     1 - nt & \text{ if } t \in [0,1/n], \\
     0 & \text{ if } t \geq 1/n.
 \end{cases}$$
\end{remark}

\begin{remark}
    The existence of the random variable $z$ from equation \eqref{zdef} has already been established by \cite{Stenflo2012}. Our argument is also inspired by the proof of the existence of Furstenberg measures exposed in \cite{BougerolLacroix1985}*{Chapter II}.
\end{remark}

To complete the proof of Theorem~\ref{HutchinsonSuperGen1}, it remains to show the estimate \eqref{Concentration}. To do so, we establish that $\mu^{*n}*\delta_x$ converges to $\nu$  exponentially fast. For a function $f: X \to \R$, we denote by $||f||_{\infty} = \sup_{x \in X} |f(x)|$ and by $\mathrm{Lip}(f) = \sup_{x,y \in X} \frac{|f(x) - f(y)|}{d(x,y)}$ the Lipschitz constant of $f$.

\begin{lemma}\label{ExponentialConvergence}
    Let $\mu$ be a probability measure on $L(X)$ with $\chi_{\mu} < 0$ and satisfying $\mathbb{E}_{g\sim \mu}[\rho(g)^t] < \infty$ for all $t \in \R$ and let $A_x$ be as above. Then there is $\theta > 0$ such that for a bounded Lipschitz function $f:X\to \R$, $x \in X$ and $n \geq 1$, $$\bigg|\int f(z) \, d\nu(z) - \int f(gx) \, d\mu^{*n}(g)\bigg| \ll_{\mu} A_x \max\{||f||_{\infty}, \mathrm{Lip}(f)  \}e^{-\theta n}.$$ 
\end{lemma}

\begin{proof}
    We continue with the notation from the proof of existence and uniqueness and let $\lambda \in (0,1)$ be as in Lemma~\ref{lambdaLemma}. Denote by $F_n$ the event that $\{ \rho(\gamma_1)\cdots \rho(\gamma_n) \leq \lambda^n \}$ and by $E_n  = \bigcap_{k = n}^{\infty} F_k$. Then, by Lemma~\ref{rhoLDP}, for some $\delta_1 > 0$ we have $\mathbb{P}[E_n^c] \ll_{\mu} e^{-\delta_1 n}$ . Indeed, this follows by using a suitable $\lambda \in (0,1)$ and the resulting $\delta > 0$ from Lemma~\ref{rhoLDP} as well as the union bound $\mathbb{P}[E_n^c] \leq \sum_{k \geq n} \mathbb{P}[F_n] \ll_{\mu} \sum_{k \geq n} e^{-\delta n} \ll_{\mu} e^{-\delta_1 n}$.
    
    We have shown in \eqref{CauchyConvergence} that for $\omega \in E_n$, $$d(z(\omega), z_n(x,\omega)) \leq \frac{A_x}{1-\lambda} \lambda^n.$$ To conclude, 
    \begin{align*}
        \bigg|\int f(z) \, d\nu(z) - \int f(gx) \, d\mu^{*n}\bigg| &\leq \int |f(z(\omega)) - f(z_n(x,\omega))| \, d\mathbb{P}(\omega) \\
        &=  \int_{E_n} |f(z(\omega)) - f(z_n(x,\omega))| \, d\mathbb{P}(\omega) \\ &+ \int_{E_n^c} |f(z(\omega)) - f(z_n(x,\omega))| \, d\mathbb{P}(\omega)  \\
        &\leq \frac{A_x}{1-\lambda} \lambda^n \mathrm{Lip}(f) + 2\mathbb{P}[E_n^c]||f||_{\infty},
    \end{align*} showing the claim for a sufficiently small $\theta$ such that $\max\{e^{-\delta_1}, \lambda \} \leq e^{-\theta}$.
\end{proof}

\begin{proof}(proof of Theorem~\ref{HutchinsonSuperGen1})
    Let $x\in X$ and $R > 0$. We first apply Lemma~\ref{ExponentialConvergence} to a suitable function. Let $F_R: \R\to \R$ be the function that is $0$ on $[-R/2,R/2]$, $1$ on $[-R,R]^C$ and the linear interpolation between these intervals. Then we consider the function on $X$ defined for $y \in X$ as $$f_{R,x}(y) = F_R(d(y,x)).$$ Note that $f_{R,x}$ is Lipschitz with $\mathrm{Lip}(f_{R,x}) \leq 2R^{-1}$. Thus by applying Lemma~\ref{ExponentialConvergence} for $n\geq 1$,
    \begin{align}\label{AppliedEffEqui}
        \nu(B_R(x)^c) \leq \int f_{R,x} \, d\nu \leq (\mu^{*n}*\delta_x)(B_{R/2}(x)^c) + O_{\mu}(A_xR^{-1}e^{-\theta n}).
    \end{align}
    We next give a suitable bound for $(\mu^{*n}*\delta_x)(B_{R/2}(x)^c)$. If $\gamma_1, \ldots, \gamma_n \in \mathrm{supp}(\mu)$ then 
    \begin{align*}
        d(x,\gamma_1\cdots \gamma_n x) &\leq \sum_{i = 1}^n d(\gamma_1\cdots \gamma_{i-1}x, \gamma_1\cdots \gamma_ix) \\ &\leq A_x(1 + \rho(\gamma_1) + \ldots + \rho(\gamma_1\cdots \gamma_{n-1})). 
    \end{align*} Let $\eps \in (0,1)$ be a fixed constant, $\lambda \in (0,1)$ be as in Lemma~\ref{lambdaLemma} and $E_n$ be as in the proof of Lemma~\ref{ExponentialConvergence}. Then note that (using notation from the proof of Lemma~\ref{ExponentialConvergence})
    \begin{align*}
        (\mu^{*n}*\delta_x)(B_{R/2}(x)^c) &= \int 1_{B_{R/2}(x)^c}(z_{n}(x,\omega)) \, d\mathbb{P}(\omega) \\
        &\leq \int_{E_{\lfloor \eps n \rfloor}} 1_{B_{R/2}(x)^c}(z_{n}(x,\omega)) \, d\mathbb{P}(\omega) + \mathbb{P}[E_{\lfloor \eps n \rfloor}^c]
    \end{align*} If $\omega \in E_{\lfloor \eps n \rfloor}$  then for all $m \geq \lfloor\eps n \rfloor$ it holds that $\rho(g_1\cdots g_m) \leq \lambda^m$. Thus for such an $\omega$ and $\rho_{\mathrm{sup}} = \sup_{g \in \mathrm{supp}(\mu)} \rho(g)$, 
    \begin{align*}
        d(x,z_n(x,\omega)) &\leq A_x(1 + \rho(\gamma_1) + \ldots + \rho(\gamma_1\cdots \gamma_{\lfloor n\eps \rfloor - 1}) + \lambda^{\lfloor n\eps \rfloor} + \ldots \lambda^n ) \\ 
        &\leq A_x\left(\frac{1}{1-\lambda} + n\eps \max\{ 1,  \rho_{\mathrm{sup}} \}^{n\eps} \right) \\
        &\leq D_1 A_x(1 + D_2^{n\eps})
    \end{align*}
    for suitably large constants $D_1$ and $D_2$ depending on $\mu$ and $\eps$. Letting $R$ being sufficiently large and choosing $n$ such that $4D_1 A_x(1 + D_2^{n\eps}) \leq R \leq 4D_1 A_x(1 + D_2^{(n + 1)\eps})$, or equivalently $n \asymp_{\mu, x, \eps} \log R$, it therefore follows that $$(\mu^{*n}*\delta_x)(B_{R/2}(x)^c) \leq \mathbb{P}[E_{\lfloor \eps n \rfloor}^c] \ll_{\mu} e^{-\delta_1 \eps n},$$ using that $\mathbb{P}[E_n^c] \ll_{\mu} e^{-\delta_1 n}$ as in the proof of Lemma~\ref{ExponentialConvergence}. Combining the latter with \eqref{AppliedEffEqui}, we conclude that $$\nu(B_R(x)^c) \ll_{\mu} A_x R^{-1}e^{-\theta n} + e^{-\delta_1 \eps n} \ll_{\mu} A_x R^{-1}R^{-O_{\mu,\eps}(\theta)} + R^{-O_{\mu, \eps}(\delta_1)} \ll_{\mu,x, \eps} R^{-\alpha}$$ for a suitable constant $\alpha>0$. Since $\eps$ is fixed, the claim follows. 
\end{proof}

\begin{proof}(of Theorem~\ref{HutchinsonLyapunovGen})
    The proof is identical to the one of Theorem~\ref{HutchinsonSuperGen1}, using directly the large deviation assumption instead of Lemma~\ref{lambdaLemma}.
\end{proof}

    \section{Further Remarks}

    \label{section:FurtherRemarks}

\subsection{Finiteness of Differential Entropy}

Let $\nu$ be an absolutely continuous measure on $\R^d$ with density $f_{\nu} \in L^1(\R^d)$. Then the differential entropy of $\nu$ is defined as 
\begin{equation}\label{Def:DiffEnt}
    H(\nu) = \int h(f_{\nu}) \,\, d\Haarof{\R^d} \quad\quad \text{ for } \quad\quad  h(x) = \begin{cases}
        -x\log x & \text{if } x > 0 \\
        0 & \text{otherwise}
    \end{cases} 
\end{equation} and $\Haarof{\R^d}$ the Lebesgue measure on $\R^d$. We next prove the Lemma~\ref{FiniteDiffEnt}, which shows that stationary measures on $\R^d$ with a polynomial tail decay have finite differential entropy. Variants of this result are used in \cite{KittleKoglerAC} and \cite{KittleKoglerDim}. 

To give a contracting on average example where Lemma~\ref{FiniteDiffEnt} applies, consider for $x \in \R$ and $q \geq 1$ the similarities $$g_1(x) = \frac{q}{q - 1}x + 1 \quad\quad \text{ and } \quad\quad g_2(x) = \frac{q}{q + 3}x -1.$$ Then by Corollary 1.10 of \cite{KittleKoglerAC}, if $q$ is a sufficiently large prime, the self-similar measure of $\frac{1}{3}\delta_{g_1} + \frac{2}{3}\delta_{g_2}$ is absolutely continuous and so Lemma~\ref{FiniteDiffEnt} applies by Theorem~\ref{HutchinsonSuperGen1}.

\begin{lemma}\label{FiniteDiffEnt}
    Let $\nu$ be an absolutely continuous probability measure on $\R^d$. Assume that for some $\alpha > 0$ as $R \to \infty$, $$\nu(\{ x\in \R^d \,:\, |x| \geq R \}) \ll_{\nu} R^{-\alpha}.$$ Then $\nu$ has finite differential entropy. 
\end{lemma}

\begin{proof}
    Let $f_{\nu}$ be the density of $\nu$. Denote for $R > 0$ by $B_R$ the open $R$-ball around $0$ in $\R^d$. Let $L > 1$ be a sufficiently large constant and for $i = 0,1,2,\ldots $  write $p_i = \nu(B_{L^{i+1}}\backslash B_{L^{i}})$ such that, by assumption, $p_i \leq \nu(\R^d \backslash B_{L^i}) \leq L^{-i\alpha}$ for sufficiently large $i$. Thus it holds for $h(x)$ from \eqref{Def:DiffEnt} and $m_i = \Haarof{\R^d}(B_{L^{i+1}}\backslash B_{L^{i}})$ and using the convention that $0 \log 0 = 0$,
    \begin{align*}
        H(\nu) &= \sum_{i \geq 0} \int_{B_{L^{i+1}}\backslash B_{L^{i}}} -f_{\nu} \log f_{\nu} \, d\Haarof{\R^d} \\
        &= \sum_{i \geq 0} \int_{B_{L^{i+1}}\backslash B_{L^{i}}} -f_{\nu} \log \left( \frac{f_{\nu}m_i}{m_i} \right) \, d\Haarof{\R^d} \\
        &= \sum_{i \geq 0} \left(\int h(f_{\nu}m_i) \frac{1_{B_{L^{i+1}}\backslash B_{L^{i}}}}{m_i} \, d\Haarof{\R^d} + p_i\log m_i \right) \\
        &\leq \sum_{i \geq 0} h(p_i) +  p_i\log m_i,
    \end{align*} by using Jensen's inequality for the concave function $h$ and with the probability measure $\frac{1_{B_{L^{i+1}}\backslash B_{L^{i}}}}{m_i} \, d\Haarof{\R^d}$ in the last line. As $h(x)$ increases monotonically for small $x$, we have $p_i \leq L^{-i\alpha}$ and $h(p_i) \leq h(L^{-i\alpha})$ for $i \geq I$ with $I$ sufficiently large. Since moreover $m_i\leq L^{d(i+ 1)}$, the claim of the lemma follows by  
    \begin{align*}
        H(\nu) &\leq  \sum_{i \geq 0} h(p_i) +  p_i\log m_i \\
        &\leq \left(\sum_{0 \leq i \leq I} h(p_i) +  p_i\log m_i \right) + \left(\sum_{i \geq I} h(L^{-i\alpha}) + L^{-i\alpha}\log m_i \right) \\
        &\leq \left(\sum_{0 \leq i \leq I} h(p_i) +  p_i\log m_i \right) + \left(\sum_{i \geq I} (i\alpha + d(i + 1)) L^{-i\alpha}\log L  \right)\\
        &< \infty.
    \end{align*} 

\end{proof}

\subsection{A compactly supported sequence in $L(X)$}\label{L(X)example}

We recall that $L(X)$ is endowed with the compact open topology, i.e., the topology of uniform convergence on compact spaces. In this subsection, we show that even on $X = \R$ there are examples where $\mu$ is compactly supported on $L(X)$ with $\chi_{\mu} < 0$ while $\mathbb{E}_{g\sim \mu}[\rho(g)^t] = \infty$ for all $t > 0$. The latter shows, as mentioned in the introduction, that the assumption that $\mu$ being compactly supported does not imply that $\mathbb{E}_{g \sim \mu}[\rho(g)^t] < \infty$ for any $t\neq 0$. 

For example, given a sequence of real numbers $(a_n)_{n \geq 1}$ consider the following sequence of Lipschitz functions defined for $x \in \R^d$ as $$f_n(x) = \begin{cases}
    0 & \text{if } x \leq n, \\
    a_n(x-n) &\text{if } x \geq n.
\end{cases}$$ 

We note that $\rho(f_n) = a_n$ and that $f_n$ converges in $L(X)$ to the constant $0$ function, which implies that the probability measure $$\mu = \sum_{n \geq 1} p_n \delta_{f_n}$$ is compactly supported for any probability vector $(p_1, p_2, \ldots)$.

Concretely, we now take for $n \geq 1$, $$p_n = \frac{6}{\pi^2}\frac{1}{n^2} \quad\quad \text{ and } \quad\quad a_n = \begin{cases}
    \frac{1}{n} & \text{if } n \text{ is not a square number}, \\
    \sqrt{n}^{\sqrt{n}} & \text{if } n  \text{ is a square number}.\end{cases}$$
The reader may then check that $\chi_{\mu} \in (-\infty, 0)$ and $\mathbb{E}_{g\sim \mu}[\rho(g)^t] = \infty$ for all $t \neq 0$, while the unique stationary measure is $\delta_0$.

\subsection{Polynomial Lower Bound for Self-Similar Measures} \label{section:PolynomialLower}

Denote by $$\mathrm{Sim}(\R^d) < \mathrm{Aff}(\R^d)$$ the group of similarities of $\R^d$, where an affine map $g \in \mathrm{Aff}(\R^d)$ is a similarity if there exists $\rho(g) \in \R_{>0}$ and $U(g) \in \mathrm{O}(d)$ (the orthogonal matrices) such that $A(g) = \rho(g)U(g)$. Then $||A(g)|| = \rho(g)$ and we note that $\lambda_{\mu} = \chi_\mu$ for a probability measure $\mu$ on $\mathrm{Sim}(\R^d)$.

Let $\mu$ be a finitely supported probability measure on $\mathrm{Sim}(\R^d)$ with $\chi_{\mu} < 0$. Then the associated stationary measure $\nu$ is called a \textbf{self-similar measure}. To show that Theorem~\ref{HutchinsonSuperGen1} is sharp for self-similar measures, we give a polynomial lower bound on $\nu(B_R^c)$ under the assumption that $\mu$ is finitely supported with $\chi_{\mu} < 0$ and that there is $g \in \mathrm{Sim}(\R^d)$ such that $\rho(g) > 1$. The latter situation happens often as for example when $\mu = \frac{1}{2}\delta_{g_1} + \frac{1}{2}\delta_{g_2}$ and say $\rho(g_1) = 
\frac{1}{e}$ then we can choose $\rho(g_2) \in (1,e)$.

In contrast, by Hutchinson's theorem \cite{Hutchinson1981}, if $\rho(g) < 1$ for all $g \in \mathrm{supp}(\mu)$, then $\nu$ is compactly supported. Furthermore, if we only assume that $\rho(g) \leq 1$ for all $g \in \mathrm{supp}(\mu)$, then the support of $\nu$ may be compact or non-compact as the following two examples show. Indeed, if $\mu = \frac{1}{2}(\delta_{g_1} + \delta_{g_2})$ with $g_1(x) = \frac{1}{2}x + 1$ and $g_2(x) = -x$ for $x\in \R$, then $\nu$ has compact support. On the other hand, when we change $g_2$ to $g_2(x) = x + 1$ for $x\in \R$, the support of $\nu$ is non-compact. Therefore, the requirement $\rho(g) > 1$ for some $g \in \mathrm{supp}(\mu)$ is necessary for Lemma~\ref{partexpss} to hold.

\begin{lemma}\label{partexpss}
    Let $\mu$ be a probability measure on $\mathrm{Sim}(\R^d)$ supported on finitely many similarities without a common fixed point. Assume that $\chi_{\mu} < 0$ and that there is $g \in \mathrm{Sim}(\R^d)$ such that $\rho(g) > 1$. Then there is $\alpha_1 = \alpha_1(\mu) > 0$ such that for $x \in \R^d$, 
    \begin{equation}\label{non-compactsupp}
        \nu(B_R(x)^c) \gg_{\mu,x} R^{-\alpha_1}.
    \end{equation}
\end{lemma}

\begin{proof} We note that it suffices to prove the claim for a fixed $x \in \R^d$. Write $\mu = \sum_{g} p_g \delta_g$ and let $g_0 \in \mathrm{supp}(\mu)$ be a map with $\rho(g_0) > 1$ and $p_{g_0} > 0$. For convenience write $\rho_0 = \rho(g_0)$, $p_0 = p_{g_0}$ and denote by $x_0$ the unique fixed point of $g_0$, which exists since $g_0^{-1}$ is contractive and has a unique fixed point. Since the support of $\nu$ contains at least two points, we may choose $r > 0$ depending on $\mu$ such that $\nu(B_r(x_0)^c) > 0$ and note that $g_0^n B_r(x_0)^c \subset B_{\rho_0^n r}(x_0)^c$. Therefore 
    \begin{align*}
        \nu(B_{\rho_0^n r}(x_0)^c) &= (\mu^{*n} * \nu)(B_{\rho_0^n r}(x_0)^c) \\ &\geq (p_0^n \delta_{g_0^n} * \nu)(B_{\rho_0^n r}(x_0)^c) \\
        &\geq p_0^n \nu(g_0^{-n}B_{\rho_0^n r}(x_0)^c) \\
        &\geq p_0^n \nu(B_r(x_0)^c).
    \end{align*} To prove the claim, let $R \geq 1$ be such that $\rho_0^{n}r \leq R \leq \rho_0^{n + 1} r$ for an integer $n \geq 0$. In particular $n\log \rho_0 + \log r\leq \log R \leq 2n \log \rho_0$ for sufficiently large $R$. Setting $\alpha_1 = - \frac{\log p_0}{\log \rho_0} > 0$,
    \begin{align*}
        \nu(B_R(x_0)^c) \geq \nu(B_{\rho_0^n r}(x_0)^c) \gg_{\mu} p_0^n = e^{(\log p_0)n} \gg_{\mu} e^{-\alpha_1 \cdot (n \log \rho_0)} \gg_{\mu} R^{-\alpha_1}.
    \end{align*} 
\end{proof}

\bibliography{referencesgeneral.bib}
\end{document}